\documentclass[runningheads]{llncs}
\usepackage{graphicx}
\usepackage[colorlinks=true,urlcolor=blue,linkcolor=blue,citecolor=blue]{hyperref}
\usepackage{amsmath,amsfonts,amssymb}
\usepackage{doi}
\usepackage{tikz,tikz-3dplot}
\usepackage{esvect}
\usetikzlibrary{graphs}
\usetikzlibrary{graphs.standard}
\newcommand{\R}{\mathbb{R}}

\DeclareMathOperator\conv {conv} 
\DeclareMathOperator\TSP {TSP} 
\DeclareMathOperator\ATSP {ATSP}
\DeclareMathOperator\PSB {PSB}
\newcommand{\lvv}[1]{\reflectbox{\ensuremath{\vv{\reflectbox{\ensuremath{#1}}}}}}
\begin{document}
\title{On vertex adjacencies in the polytope of pyramidal tours with step-backs}
%
%
\author{Andrei Nikolaev\inst{1}\orcidID{0000-0003-4705-2409}}
\authorrunning{A. Nikolaev}
%
\institute{P.G. Demidov Yaroslavl State University, Yaroslavl, Russia
\email{andrei.v.nikolaev@gmail.com}}
\maketitle              
\begin{abstract}
We consider the traveling salesperson problem in a directed graph. 
The pyramidal tours with step-backs are a special class of Hamiltonian cycles for which the traveling salesperson problem is solved by dynamic programming in polynomial time.
The polytope of pyramidal tours with step-backs $\PSB (n)$ is defined as the convex hull of the characteristic vectors of all possible pyramidal tours with step-backs in a complete directed graph.
The skeleton of $\PSB (n)$ is the graph whose vertex set is the vertex set of $\PSB (n)$ and the edge set is the set of geometric edges or one-dimensional faces of $\PSB (n)$.
The main result of the paper is a necessary and sufficient condition for vertex adjacencies in the skeleton of the polytope $\PSB (n)$ that can be verified in polynomial time.

\end{abstract}
\section{Introduction}

We consider a classic asymmetric traveling salesperson problem: for a given complete weighted digraph $D_n = (V,E)$ it is required to find a Hamiltonian tour of minimum weight.
We denote by $HT_{n}$ the set of all Hamiltonian tours in $D_{n}$.
With each Hamiltonian tour $x \in HT_{n}$ we associate a characteristic vector $x^{v} \in \R^{E}$ by the following rule:
\[
x^v_e = 
\begin{cases}
1,& \text{ if an edge } e \in E \text{ is contained in the tour } x,\\
0,& \text{ otherwise.}
\end{cases}
\]
The polytope
\[\ATSP(n) = \conv \{x^v \ | \ y \in HT_n \}\]
is called \textit{the asymmetric traveling salesperson polytope}.

The \textit{skeleton} of a polytope $P$ (also called \textit{$1$-skeleton}) is the graph whose vertex set is the vertex set of $P$ (characteristic vectors $x^{v}$ for the traveling salesperson problem) and edge set is the set of geometric edges or one-dimensional faces of $P$.
Many papers are devoted to the study of 1-skeletons associated with combinatorial problems.
On the one hand, the vertex adjacencies in 1-skeleton is of great interest for the development of algorithms to solve problems based on local search technique (when we choose the next solution as the best one among adjacent solutions).
For example, various algorithms for perfect matching, set covering, independent set, ranking of objects, problems with fuzzy measures, and many others are based on this idea \cite{Aguilera,Balinski,Chegireddy,Combarro,Gabow}.
On the other hand, some characteristics of 1-skeleton of the problem, such as the diameter and the clique number, estimate the time complexity for different computation models and classes of algorithms \cite{Bondarenko-Maksimenko,Bondarenko-Nikolaev-2016,Bondarenko-Shovgenov-BS,Grotschel}.

However, for such combinatorial problems as a knapsack, a set partition and set covering, an integer programming, a leaf-constrained and degree-constrained minimum spanning tree, a connected $k$-factor and a some others already the question whether two vertices in 1-skeleton are adjacent is an NP-complete problem \cite{Bondarenko-Shovgenov-ST,Matsui,Simanchev}.
Historically, the first result of this type was obtained by Papadimitriou for the traveling salesperson polytope.

\begin{theorem} [Papadimitriou, \cite{Papadimitriou}]
	The question whether two vertices of the polytope $\TSP(n)$ are nonadjacent is NP-complete.
\end{theorem}

In this regard, the study of 1-skeleton of the traveling salesperson problem has shifted to the study of individual faces of the polytope \cite{Sierksma1993,Sierksma1995}, the polytopes of related problems \cite{Arthanari}, as well as the polytopes of special cases of the traveling salesperson problem.
In particular, for the polytope of the pyramidal tours it was established that the verification of the vertex adjacency in 1-skeleton can be performed in linear time \cite{Bondarenko-Nikolaev-2017,Bondarenko-Nikolaev-2018}.

In this paper, we consider a 1-skeleton of a wider class of the pyramidal tours with step-backs.

\section{Pyramidal tours with step-backs}

We suppose that the cities are labeled from $1$ to $n$.
Let $\tau$ be a Hamilton tour.
We denote the successor of $i$-th city as $\tau (i)$.
For any natural $k$, we denote the $k$-th successor of $i$ as $\tau^{k}(i)$, the $k$-th predecessor of $i$ as $\tau^{-k}(i)$.

The city $i$ satisfying $\tau^{-1}(i) < i$ and $\tau(i) < i$ is called a \textit{peak}.

A \textit{pyramidal tour} is a Hamiltonian tour with only one peak  $n$.

A \textit{step-back peak} (Fig. \ref{Fig_step_backs}) is the city $i$, such that
\[\tau^{-1} < i, \ \tau(i) = i - 1 \text{ and } \tau^{2}(i) > i,\]
or
\[\tau^{-2} > i, \ \tau^{-1}(i) = i - 1 \text{ and } \tau(i) < i.\]

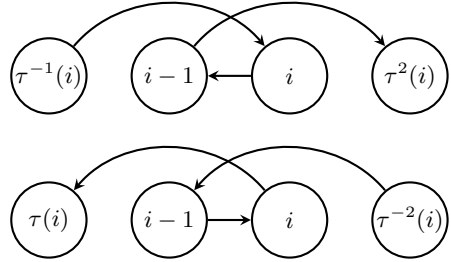
\begin{figure} [t]
	\centering
	\begin{tikzpicture}[scale=0.8]
	\begin{scope}[every node/.style={circle,thick,draw,inner sep=0pt, minimum size=1cm}]
	\node (A) at (0,0) {$\tau^{-1}(i)$};
	\node (B) at (2,0) {$i-1$};
	\node (C) at (4,0) {$i$};
	\node (D) at (6,0) {$\tau^{2}(i)$};
	
	\end{scope}
	\draw [thick,->,>=stealth] (A) edge [bend left=50] (C);
	\draw [thick,->,>=stealth] (C) edge (B);
	\draw [thick,->,>=stealth] (B) edge [bend left=50] (D);
	\end{tikzpicture}
	\\[2mm]
	\begin{tikzpicture}[scale=0.8]
	\begin{scope}[every node/.style={circle,thick,draw,inner sep=0pt, minimum size=1cm}]
	\node (A) at (0,0) {$\tau(i)$};
	\node (B) at (2,0) {$i-1$};
	\node (C) at (4,0) {$i$};
	\node (D) at (6,0) {$\tau^{-2}(i)$};
	
	\end{scope}
	\draw [thick,<-,>=stealth] (A) edge [bend left=50] (C);
	\draw [thick,<-,>=stealth] (C) edge (B);
	\draw [thick,<-,>=stealth] (B) edge [bend left=50] (D);
	\end{tikzpicture}
	
	\caption {Step-back in ascending and descending order}
	\label {Fig_step_backs}
\end{figure}

A \textit{proper peak} is a peak $i$ which is not a step-back peak.
A \textit{pyramidal tour with step-backs} is a Hamiltonian tour with exactly one proper peak $n$.

Traveling salesperson problem on pyramidal tours is one of the most studied polynomial special case of the problem \cite{Gilmore}.
A more general class of pyramidal tours with step-backs was introduced in \cite{Enomoto}.
These tours are of interest, since, on the one hand, the minimum cost pyramidal tour with step-backs can be found in $O(n^2)$ time by dynamic programming, and, on the other hand, there are known restrictions on the distance matrix that guarantee the existence of an optimal tour that is pyramidal with step-backs \cite{Enomoto}.

A generalization of pyramidal tours with step-backs is the class of quasi-pyramidal tours for which the traveling salesperson problem is fixed parameter tractable \cite {Khachay,Oda}.

We denote by $PSBT_n$ the set of all pyramidal tours with step-backs in the complete digraph $D_n = (V,E)$.
With each pyramidal tour with step-backs $x \in PSBT_n$ we associate a characteristic vector $x^v \in \R^E$ by the following rule:
\[
x^v_e = 
\begin{cases}
1,& \text{ if an edge } e \in E \text{ is contained in the tour } x,\\
0,& \text{ otherwise.}
\end{cases}
\]
The polytope
\[\PSB(n) = \conv \{x^v \ | \ x \in PSBT_n \}\]
is called the \textit{polytope of pyramidal tours with step-backs}.

Besides we use a special encoding to represent the pyramidal tours with step-backs.
With each tour $x \in PSBT_n$ we associate an encoding vector $x^{0,1,sb}$ of length $n-2$, each coordinate corresponds to the city number, the positions of the cities $1$ and $n$ are fixed, by the following rule:
\begin{gather*}
\forall i \ (2 \leq i \leq n-1)\\
x^{0,1,sb}_{i} =
\begin{cases}
1, &\text{ if } i \text{ is visited by } x \text{ in ascending order},\\
\overline{1}, &\text{ if } i \text{ is the beginning of ascending step-back},\\
\lvv{1}, &\text{ if } i \text{ is the ending of ascending step-back},\\
0, &\text{ if } i \text{ is visited by } x \text{ in descending order},\\
\overline{0}, &\text{ if } i \text{ is the beginning of descending step-back},\\
\vv{0}, &\text{ if } i \text{  is the ending of ascending step-back}.
\end{cases}
\end{gather*}
An example of a pyramidal tour with step-backs and the corresponding vector $x^{0,1,sb}$ is shown in Fig.~\ref{Fig_encoding}.

\begin{figure}[t]
	\centering
	\begin{tikzpicture}[scale=1.1]
	\begin{scope}[every node/.style={circle,thick,draw,inner sep=3pt}]
	\node (A) at (0,0) {1};
	\node (B) at (1,0) {2};
	\node (C) at (2,0) {3};
	\node (D) at (3,0) {4};
	\node (E) at (4,0) {5};
	\node (F) at (5,0) {6};
	\node (G) at (6,0) {7};
	\node (H) at (7,0) {8};
	\end{scope}
	\draw [thick,->,>=stealth] (A) edge (B);
	\draw [thick,->,>=stealth] (B) edge [bend left=45] (E);
	\draw [thick,->,>=stealth] (E) edge (D);
	\draw [thick,->,>=stealth] (D) edge [bend left=45] (G);
	\draw [thick,->,>=stealth] (G) edge (H);
	\draw [thick,->,>=stealth] (H) edge [bend left=50] (F);
	\draw [thick,->,>=stealth] (F) edge [bend left=50] (C);
	\draw [thick,->,>=stealth] (C) edge [bend left=50] (A);
	
	\node at (3.5,-1.3) {$\Updownarrow$};
	\node at (3.5,-2) {$x^{0,1,sb} = \left\langle 1,0,\lvv{1,1},0,1 \right\rangle$};
	
	\begin{scope}[every node/.style={circle,inner sep=1pt}]
	\node (J) at (2,-0.7) {0};
	\node (K) at (3,1.0) {1};
	\node (L) at (4,1.0) {1};
	\node (M) at (5,-0.7) {0};
	\node (N) at (6,0.7) {1};
	\node (O) at (1,0.7) {1};
	\end{scope}
	
	\draw [->,>=stealth,line width=0.3mm] (J) edge (C);
	\draw [->,>=stealth,line width=0.3mm] (K) edge (D);
	\draw [->,>=stealth,line width=0.3mm] (L) edge (E);
	\draw [->,>=stealth,line width=0.3mm] (M) edge (F);
	\draw [->,>=stealth,line width=0.3mm] (N) edge (G);
	\draw [->,>=stealth,line width=0.3mm] (O) edge (B);
	
	\draw [<-,>=stealth,line width=0.3mm] (2.9,1.25) -- (4.1, 1.25);
	\end{tikzpicture}
	\caption {An example of a tour and corresponding encoding}
	\label {Fig_encoding}
\end{figure}
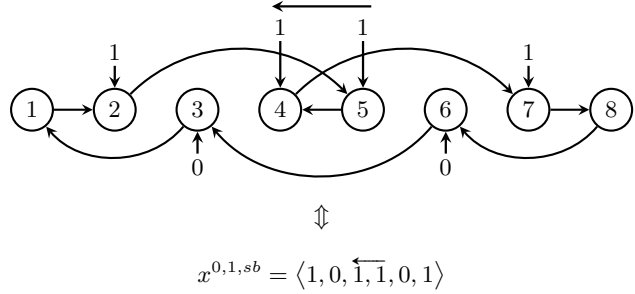

We denote by $x^{0,1,sb}_{[i,j]}$ a fragment of encoding on coordinates from $i$ to $j$.
The superscript indicates what we consider in the encoding.
For example, $x^{1,sb}_{[i,j]}$ means a fragment of the encoding only in ascending order taking into account step-backs; $x^{0,1}_{[i,j]}$ -- a fragment of the encoding without step-backs, etc.

\section{Auxiliary statements}

We denote by $x \cup y$ a multigraph that contains all edges of both tours $x$ and $y$.

\begin{lemma}[Sufficient condition for nonadjacency]\label{lemma_sufficient}
	Given two tours $x$ and $y$, if the multigraph $x \cup y$ includes a pair of other complementary pyramidal tours with step-backs, then the corresponding vertices $x^v$ and $y^v$ of the polytope $\PSB(n)$ are not adjacent.
\end{lemma}
\begin{proof}
	Let two complementary pyramidal tours with step-backs $z$ and $t$ be composed from the edges of $x$ and $y$, then for the corresponding vertices the following equality takes place:
	\begin{equation}\label{eq_sufficient_convex}
	x^v + y^v = z^v + t^v.
	\end{equation}
	
	We divide the equality (\ref{eq_sufficient_convex}) by 2 and obtain that the segment $[x^v,y^v]$ intersects with the segment $[z^v,t^v]$.
	Therefore, the vertices $x^v$ and $y^v$ of the polytope $\PSB(n)$ are not adjacent.
\end{proof}

\begin{lemma}[Necessary condition for nonadjacency]\label{lemma_necessary}
	If the vertices $x^v$ and $y^v$ of the polytope $\PSB(n)$ are not adjacent, then the multigraph $x \cup y$ includes at least two pyramidal tours with step-backs other than $x$ and $y$.
\end{lemma}
\begin{proof}
	Let the vertices $x^v$ and $y^v$ of the polytope $\PSB(n)$ be not adjacent, then the segment $[x^v,y^v]$ intersects with the convex hull of some of the remaining vertices of the polytope $\PSB(n)$:
	\begin{gather}
	\alpha x^v + (1-\alpha) y^v = \beta_1 z^v_1 + \ldots + \beta_m z^v_m,\label{eq_convex_necessary}\\
	\forall i:\ \beta_i > 0 \text{ and } \beta_1 + \ldots + \beta_m = 1.\nonumber
	\end{gather}
	Note that $m \geq 2$, since the segment connecting two vertices of a convex polytope cannot intersect a third vertex.
	If at least one tour $z_i$ includes an edge $e$ that does not belong to the multigraph $x \cup y$, then the equality (\ref{eq_convex_necessary}) is violated in the coordinate corresponding to the edge $e$.
\end{proof}

\begin{lemma}\label{lemma_zero_coefficient}
	Let $x$ and $y$ be two pyramidal tours with step-backs.
	Suppose that there are two edges $e_x$ of $x$ and $e_y$ of $y$ that no pyramidal tour with step-backs can include both edges $e_x$ and $e_y$ at the same time.
	Let the corresponding vertices $x^v$ and $y^v$ of the polytope $\PSB(n)$ be not adjacent.
	Then the convex combination of the remaining vertices of $\PSB (n)$, that coincide with the convex combination of $x^v$ and $y^v$, cannot include with a nonzero coefficient any vertex corresponding to the tour without at least one edge of the pair $e_x$ and $e_y$.
\end{lemma}
\begin{proof}
	Since the vertices $x^v$ and $y^v$ are not adjacent, their convex hull intersects with the convex hull of the remaining vertices of the polytope $\PSB (n)$:
	\begin{gather}
	\alpha x^v + (1-\alpha) y^v = \sum {\beta_i z^v_i} + \sum {\beta_{j,x} z^v_{j,x}} + \sum {\beta_{k,y} z^v_{k,y}}, \label{eq_convex_comb}\\
	\sum {\beta_i} + \sum {\beta_{j,x}} + \sum {\beta_{k,y}} = 1,\nonumber
	\end{gather}
	where $z_{j,x}$ are all pyramidal tours with step-backs, containing the edge $e_x$, and $z_{j,y}$ are all pyramidal tours with step-backs, containing the edge $e_y$.
	The remaining tours $z_i$ do not contain edges $e_x$ and $e_y$, and no pyramidal tour with step-backs can include both edges $e_x$ and $e_y$ at the same time.
	The equality (\ref {eq_convex_comb}) in the coordinates, corresponding to $e_x$ and $e_y$, takes the form of a system
	\[\begin{cases}
	\alpha = \sum {\beta_{j,x}},\\
	1 - \alpha = \sum {\beta_{k,y}}.
	\end{cases}
	\]
	Therefore, $\sum {\beta_i} = 0$.
\end{proof}

\section{Necessary and sufficient condition for adjacency}

We consider 12 blocks of the following form (a wavy line means that the corresponding coordinate can either contain a step-back or not):
\begin{gather*}
U_{11} = \left\langle \begin{matrix} 1 \\ 1 \end{matrix} \right\rangle,\
U_{00} = \left\langle \begin{matrix} 0 \\ 0 \end{matrix} \right\rangle,\
U_{1111} = \left\langle \begin{matrix} \lvv {1 \ \ 1} \\ \lvv {1 \ \ 1} \end{matrix} \right\rangle,\
U_{0000} = \left\langle \begin{matrix} \vv {0 \ \ 0} \\ \vv {0 \ \ 0} \end{matrix} \right\rangle,\\
L_{1110} = \left\langle \begin{matrix} \lvv {1 \ \ 1} \\ 1 \ \ \tilde{0} \end{matrix} \right\rangle,\
L_{1011} = \left\langle \begin{matrix} 1 \ \ \tilde{0} \\ \lvv {1 \ \ 1} \end{matrix} \right\rangle,\
L_{0001} = \left\langle \begin{matrix} \vv {0 \ \ 0} \\ 0 \ \ \tilde{1} \end{matrix} \right\rangle,\
L_{0100} = \left\langle \begin{matrix} 0 \ \ \tilde{1} \\ \vv {0 \ \ 0} \end{matrix} \right\rangle,\\
R_{1101} = \left\langle \begin{matrix} \lvv {1 \ \ 1} \\ \tilde{0} \ \ 1 \end{matrix} \right\rangle,\
R_{0111} = \left\langle \begin{matrix} \tilde{0} \ \ 1 \\ \lvv {1 \ \ 1} \end{matrix} \right\rangle,\
R_{0010} = \left\langle \begin{matrix} \vv {0 \ \ 0} \\ \tilde{1} \ \ 0 \end{matrix} \right\rangle,\
R_{1000} = \left\langle \begin{matrix} \tilde{1} \ \ 0 \\ \vv {0 \ \ 0} \end{matrix} \right\rangle.
\end{gather*}

\begin{theorem}\label{theorem_adjacency}
	Vertices $x^v$ and $y^v$ of the polytope $\PSB(n)$ are not adjacent if and only if the following conditions are satisfied.
	\begin{itemize}
		\item There exists a city $i$ (called \textit{left block}) such that the tours $x$ and $y$ on the coordinate $i$ (coordinates $i$ and $i+1$ for double blocks) have the form of $U,L$, or $i = 1$.
		
		\item There exists a city $j$ (called \textit{right block}) such that the tours $x$ and $y$ on the coordinate $j$ (coordinates $j-1$ and $j$ for double blocks) have the form of $U,R$, or $j = n$.
		
		We denote by $i_a$ the first city after the left block: $i_a = i+1 $ for single blocks and $i_a = i+2$ for double blocks.
		We denote by $j_b$ the last city before the right block: $j_b = i-1$ for single blocks and $j_b = j-2$ for double blocks.
		
		Two blocks cut the encoding of the tours into three parts: the left (less than $i_a$), the central (from $i_a$ to $j_b$) and the right (larger than $j_b$).
		
		\item In the central part, the coordinates of $x^{0,1}$ and $y^{0,1}$ completely coincide: $x^{0,1}_{[i_a,j_b]} = y^{0,1}_{[i_a,j_b]}$.
		
		We say that two tours
		\begin{itemize}
			\item differ in the left part if $x^{0,1,sb}_{[1,i_a-1]} \neq y^{0,1,sb}_{[1,i_a-1]}$,
			\item differ in the right part if $x^{0,1,sb}_{[j_b+1,n]} \neq y^{0,1,sb}_{[j_b+1,n]}$,
			\item differ in the central part in ascending order if $x^{1,sb}_{[i_a,j_b]} \neq y^{1,sb}_{[i_a,j_b]}$,
			\item differ in the central part in descending order if $x^{0,sb}_{[i_a,j_b]} \neq y^{0,sb}_{[i_a,j_b]}$.
		\end{itemize}
		
		The remaining conditions are divided into four cases depending on the values of $x^{0,1}_i$ and $x^{0,1}_j$.
		
		\begin{enumerate}
			\item If $x^{0,1}_{i} = x^{0,1}_{j} = 1$, then the tours differ
			\begin{itemize}
				\item in the central part in ascending order;
				\item in the left part, or in the central part in descending order, or in the right part.
			\end{itemize}
		
			\item If $x^{0,1}_{i} = x^{0,1}_{j} = 0$, then the tours differ
			\begin{itemize}
				\item in the central part in descending order;
				\item in the left part, or in the central part in ascending order, or in the right part.
			\end{itemize}
		
			\item If $x^{0,1}_{i} = 1, x^{0,1}_{j} = 0$, then the tours differ	
			\begin{itemize}
				\item in the central part in ascending order or in the right part;
				\item in the central part in descending order or in the left part.
			\end{itemize}
		
			\item If $x^{0,1}_{i} = 0, x^{0,1}_{j} = 1$, then the tours differ
		
			\begin{itemize}
				\item in the central part in descending order or in the right part;
				\item in the central part in ascending order or in the left part.
			\end{itemize}
		\end{enumerate}
	
	Cities $1$ and $n$ can be considered in the encoding as visited in ascending or descending order, if required.
	\end{itemize}
\end{theorem}

\begin{proof}
	\textit{Necessity.} Let the vertices $x^v$ and $y^v$ of $\PSB(n)$ be not adjacent, then by Lemma~\ref{lemma_necessary} there exists a pyramidal tour with step-backs $z \subset x \cup y$, different from $x$ and $y$, such that the vertex $z^v$ is in a convex combination
	\[\alpha x^v + (1-\alpha) y^v = \beta z^v + \sum{\beta_i z_i}\]
	with a nonzero coefficient.

	We choose the city $i$ with the smallest number such that $z$ enters $i$ along an edge of the tour $x$, and leaves along an edge of the tour $y$. 	
	We choose the city $j$ with the smallest number such that $z$ enters $j$ along an edge of the tour $y$, and leaves along an edge of the tour $x$.
	By construction, the tour $z$ contains edges of both $x$ and $y$, therefore such cities exist.
	
	\textit{Part 1.} Let us prove that the city $i$ (city $j$) cannot be visited by the tours $x$ and $y$ in opposite orders.
	Without loss of generality, we consider the case when the city $i$ is visited by $z$ and $x$ in ascending order, and by $y$ in descending order.
	The remaining cases are treated similarly, since they are completely symmetrical.
	
	Let's consider an edge $e_{i,y}$ of the tour $y$ that leaves $i$ and is included in the tour $z$. It can lead either to a city with a smaller number, or to a city with a larger number.
	
	\begin{enumerate}
		\item Let the edge $e_{i,y}$ lead to a city with a smaller number. 
		However, the edge $e_{i,y}$ is a part of $z$ in ascending order.
		The only possible option is an edge of the form $(i-1) \leftarrow (i)$ that will be a step-back in ascending order of the tour $z$.
		In the multigraph $x \cup y$ there are two edges leaving $i-1$: $e_{i-1,x}$ of $x$ and $e_{i-1,y}$ of $y$.
		The edge $e_{i-1,y}$ leads to a city with a smaller number, since $i-1$ is visited by $y$ in descending order.
		Thus, it cannot be a part of $z$, otherwise $z$ in not a pyramidal tour with step-backs.
		Therefore, $z$ can go only along the edge $e_{i-1,x}$ to a city with a larger number.
		Since the cities $i-1$ and $i$ are visited by $x$ in ascending order, only two configurations are possible (the city $i$ is in bold):
			\[
		\text{a)} \left\langle \begin{matrix} 1 \ \ \bf{1} \\ 0 \ \ \bf{0} \end{matrix} \right\rangle,\
		\text{b)} \left\langle \begin{matrix} 1 \ \ \lvv {{\bf 1} \ \ 1} \\ 0 \ \ \bf{0} \ \ ~ \end{matrix} \right\rangle.
		\]
		In both of them, the tour $z$ goes to a city that has already been passed before: a) $i$, b) $i + 1$ (Fig.~\ref{Fig_1_0_back}). 
		Hereinafter, if there are no additional instructions, the following notation is used in the figures: solid edges -- edges of $z$, dashed -- edges of $(x \cup y) \backslash z$, dotted -- transitions of $z$ between edges of $x$ and $y$.
		
		\begin{figure} [t]
			\centering
			\begin{tikzpicture}[scale=0.7]
			
			\begin{scope}[every node/.style={circle,thick,draw,inner sep=0pt, minimum size=0.9cm}]
			\node (A) at (0,0) {$i-1$};
			\node (B) at (2,0) {$i$};
			\node (C) at (0,-2) {$i-1$};
			\node (D) at (2,-2) {$i$};	
			\end{scope}
			
			\node at (-1.2,0) {$x$};
			\node at (-1.2,-2) {$y$};
			\node at (-2,-1) {a)};
			
			\draw [thick,->,>=stealth] (A) edge (B);
			\draw [thick,->,>=stealth,dotted] (B) edge (D);
			\draw [thick,->,>=stealth] (D) edge (C);
			\draw [thick,->,>=stealth,dotted] (C) edge (A);
			
			\begin{scope}[xshift=6cm]
			\begin{scope}[every node/.style={circle,thick,draw,inner sep=0pt, minimum size=0.9cm}]
			\node (A) at (0,0) {$i-1$};
			\node (B) at (2,0) {$i$};
			\node (C) at (4,0) {$i+1$};
			
			\node (D) at (0,-2) {$i-1$};
			\node (E) at (2,-2) {$i$};	
			\end{scope}
			
			\node at (-1.2,0) {$x$};
			\node at (-1.2,-2) {$y$};
			\node at (-2,-1) {b)};
			
			\draw [thick,->,>=stealth] (A) edge [bend left=50] (C);
			\draw [thick,->,>=stealth] (C) edge (B);
			\draw [thick,->,>=stealth,dotted] (B) edge (E);
			\draw [thick,->,>=stealth] (E) edge (D);
			\draw [thick,->,>=stealth,dotted] (D) edge (A);
			\end{scope}
			\end{tikzpicture}
			
			\caption {Transition $1 \rightarrow 0$ (cases 1 (a) and (b))}
			\label {Fig_1_0_back}
		\end{figure}
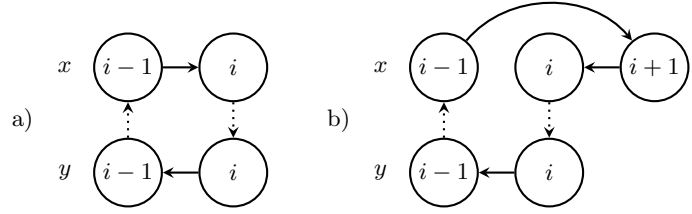
	
		\item Let the edge $e_{i,y}$ lead to a city with a larger number. However, $e_{i,y}$ is a part of $y$ in descending order. 
		Consequently, this is an edge of the form $(i) \rightarrow (i+1)$ that was a step-bask of $y$.
		The next edge of $z$ has to go to a city with a larger number, since we cannot return to $i$.
		Two edges $e_{i+1,y}$ of $y$ and $e_{i+1,x}$ of $x$ leaves the city $i-1$.
		The edge $e_{i+1,y}$ is directed to a city with a smaller number, since $i+1$ is visited by $y$ in descending order. Therefore, $z$ can go only along the edge $e_{i+1,x}$.
		We consider the possible configurations:
		\begin{gather*}
		\text{a)} \left\langle \begin{matrix} {\bf 1} \ \ 1 \\ \vv {{\bf 0} \ \ 0} \end{matrix} \right\rangle,\
		\text{b)} \left\langle \begin{matrix} {\bf 1} \ \ \lvv {1 \ \ 1} \\ \vv {{\bf 0} \ \ 0} \ \ ~ \end{matrix} \right\rangle,\
		\text{c)} \left\langle \begin{matrix} {\bf 1} \ \ \vv {0 \ \ 0} \\ \vv {{\bf 0} \ \ 0} \ \ ~ \end{matrix} \right\rangle,\\
		\text{d)} \left\langle \begin{matrix} \lvv {1 \ \ {\bf 1}} \ \ 1 \\ ~ \ \ \vv {{\bf 0} \ \ 0} \end{matrix} \right\rangle,\
		\text{e)} \left\langle \begin{matrix} \lvv {1 \ \ {\bf 1}} \ \ \lvv{1 \ \ 1} \\~ \ \ \vv {{\bf 0} \ \ 0} \ \ ~ \end{matrix} \right\rangle,\
		\text{f)} \left\langle \begin{matrix} \lvv {1 \ \ {\bf 1}} \ \ \vv{0 \ \ 0} \\~ \ \ \vv {{\bf 0} \ \ 0} \ \ ~ \end{matrix} \right\rangle.
		\end{gather*}
		
		\begin{enumerate}
			\item [a)] Transitions in cities $i$ and $i+1$ between tours $x$ and $y$ do not make sense, since the edge $(i) \rightarrow (i+1)$ is included in both tours (Fig.~\ref{Fig_1_0_forward_ab}).
	
		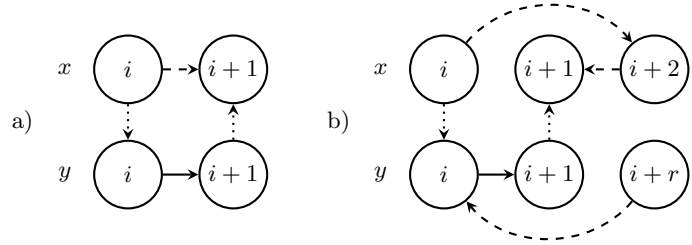
\begin{figure} [t]
			\centering
			\begin{tikzpicture}[scale=0.7]
			
			\begin{scope}[every node/.style={circle,thick,draw,inner sep=0pt, minimum size=0.9cm}]
			\node (A) at (0,0) {$i$};
			\node (B) at (2,0) {$i+1$};
			\node (C) at (0,-2) {$i$};
			\node (D) at (2,-2) {$i+1$};	
			\end{scope}
			
			\node at (-1.2,0) {$x$};
			\node at (-1.2,-2) {$y$};
			\node at (-2,-1) {a)};
			
			\draw [thick,->,>=stealth,dashed] (A) edge (B);
			\draw [thick,->,>=stealth,dotted] (A) edge (C);
			\draw [thick,->,>=stealth] (C) edge (D);
			\draw [thick,->,>=stealth,dotted] (D) edge (B);
			
			\begin{scope}[xshift=6cm]
			\begin{scope}[every node/.style={circle,thick,draw,inner sep=0pt, minimum size=0.9cm}]
			\node (A) at (0,0) {$i$};
			\node (B) at (2,0) {$i+1$};
			\node (C) at (4,0) {$i+2$};
			\node (A1) at (0,-2) {$i$};
			\node (B1) at (2,-2) {$i+1$};
			\node (C1) at (4,-2) {$i+r$};	
			\end{scope}
			
			\node at (-1.2,0) {$x$};
			\node at (-1.2,-2) {$y$};
			\node at (-2,-1) {b)};
			
			\draw [thick,->,>=stealth,dashed] (A) edge [bend left=50] (C);
			\draw [thick,->,>=stealth,dashed] (C) edge (B);
			\draw [thick,->,>=stealth] (A1) edge (B1);
			\draw [thick,->,>=stealth,dashed] (C1) edge [bend left=50] (A1);
			\draw [thick,->,>=stealth,dotted] (A) edge (A1);
			\draw [thick,->,>=stealth,dotted] (B1) edge (B);			
			\end{scope}	
			\end{tikzpicture}
			
			\caption {Transition $1 \rightarrow 0$ (cases 2 (a) and (b))}
			\label {Fig_1_0_forward_ab}
		\end{figure}
	
			\item [b)] We consider the edges $(i) \rightarrow (i+2)$ of $x$ and $(i) \leftarrow (i+r)$ of $y$ (Fig.~\ref{Fig_1_0_forward_ab}).
			None of them is included in $z$.
			The edge $(i) \rightarrow (i+2)$ cannot be a part of descending order, the edge $(i) \leftarrow (i+r)$ cannot be a part of ascending order.
			Therefore, no pyramidal tour with step-backs can contain both edges at the same time.
			By Lemma~\ref{lemma_zero_coefficient}, the vertex $z^v$ cannot be included in a convex combination that coincides with a convex combination of $x^v$ and $y^v$ with a nonzero coefficient.
			We got a contradiction.
			
			\item [c)] We consider the edges $(i-r) \leftarrow (i+2)$ of $x$ and $(i-s) \leftarrow (i+1)$ of $y$ (Fig.~\ref{Fig_1_0_forward_cd}).
			None of them is included in $z$, since the cities $i+1$ and $i+2$ are visited in ascending order.
			However, only these two edges lead from cities with numbers greater than $i$ to cities with numbers less than $i$.
			The tour $z$ cannot return to the city $1$.
			We got a contradiction.
				
			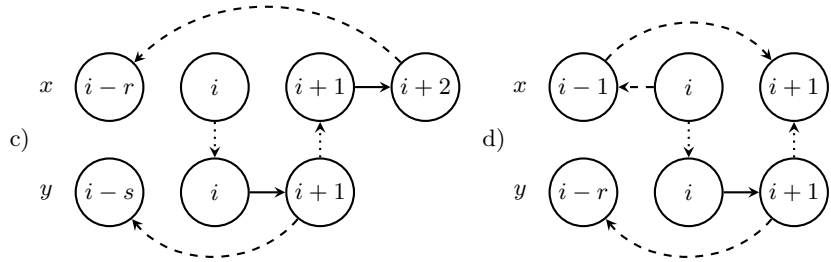
\begin{figure} [t]
				\centering
				\begin{tikzpicture}[scale=0.7]
				
				\begin{scope}[every node/.style={circle,thick,draw,inner sep=0pt, minimum size=0.9cm}]
				\node (O) at (-2,0) {$i-r$};
				\node (A) at (0,0) {$i$};
				\node (B) at (2,0) {$i+1$};
				\node (C) at (4,0) {$i+2$};
				\node (O1) at (-2,-2) {$i-s$};
				\node (A1) at (0,-2) {$i$};
				\node (B1) at (2,-2) {$i+1$};
				\end{scope}
				
				\node at (-3.2,0) {$x$};
				\node at (-3.2,-2) {$y$};
				\node at (-3.7,-1) {c)};
				
				\draw [thick,->,>=stealth] (B) edge (C);
				\draw [thick,->,>=stealth] (A1) edge (B1);
				\draw [thick,->,>=stealth,dashed] (C) edge [bend right=45] (O);
				\draw [thick,->,>=stealth,dashed] (B1) edge [bend left=50] (O1);
				\draw [thick,->,>=stealth,dotted] (A) edge (A1);
				\draw [thick,->,>=stealth,dotted] (B1) edge (B);
				
				\begin{scope}[xshift=7cm]
				\begin{scope}[every node/.style={circle,thick,draw,inner sep=0pt, minimum size=0.9cm}]
				\node (A) at (0,0) {$i-1$};
				\node (B) at (2,0) {$i$};
				\node (C) at (4,0) {$i+1$};
				\node (A1) at (0,-2) {$i-r$};
				\node (B1) at (2,-2) {$i$};
				\node (C1) at (4,-2) {$i+1$};	
				\end{scope}
				
				\node at (-1.2,0) {$x$};
				\node at (-1.2,-2) {$y$};
				\node at (-1.7,-1) {d)};
				
				\draw [thick,->,>=stealth,dashed] (A) edge [bend left=50] (C);
				\draw [thick,->,>=stealth,dashed] (B) edge (A);
				\draw [thick,->,>=stealth] (B1) edge (C1);
				\draw [thick,->,>=stealth,dashed] (C1) edge [bend left=50] (A1);
				\draw [thick,->,>=stealth,dotted] (B) edge (B1);
				\draw [thick,->,>=stealth,dotted] (C1) edge (C);					
				\end{scope}
				
				\end{tikzpicture}
				
				\caption {Transition $1 \rightarrow 0$ (cases 2 (c) and (d))}
				\label {Fig_1_0_forward_cd}
			\end{figure}
		
			\item[d)] We consider the edges $(i-1) \rightarrow (i+1)$ of $x$ and $(i-r) \leftarrow (i+1)$ of $y$ (Fig.~\ref {Fig_1_0_forward_cd}).
			As in case b), these two edges are not included in $z$ and no pyramidal tour with step-backs can contain both these edges at the same time.
			By Lemma~\ref{lemma_zero_coefficient}, we got a contradiction.

			\item [e)] The configuration has the form shown in Fig.~\ref{Fig_1_0_forward_e} (the edges of $z$ are solid).
			\begin{figure} [t]
				\centering
				\begin{tikzpicture}[scale=0.7]
				
				\begin{scope}[every node/.style={circle,thick,draw,inner sep=0pt, minimum size=0.9cm}]
				\node (A) at (0,0) {$i-1$};
				\node (B) at (2,0) {$i$};
				\node (C) at (4,0) {$i+1$};
				\node (D) at (6,0) {$i+2$};			
				\node (O1) at (-2,-2) {$i-r$};
				\node (A1) at (0,-2) {$i-1$};
				\node (B1) at (2,-2) {$i$};
				\node (C1) at (4,-2) {$i+1$};
				\node (D1) at (6,-2) {$i+2$};
				\node (E1) at (8,-2) {$i+s$};	
				\end{scope}
				
				\begin{scope}[every node/.style={circle,thick,draw,inner sep=0pt, minimum size=0.9cm}]
				\node (O) at (-2,0) {};
				\node (E) at (8,0) {};
				\end{scope}
				
				\node at (-3.2,0) {$x$};
				\node at (-3.2,-2) {$y$};
				
				\draw [thick,->,>=stealth] (O) edge [bend left=50] (B);
				\draw [thick,->,>=stealth] (B1) edge (C1);
				\draw [thick,->,>=stealth] (C) edge [bend left=50] (E);
				\draw [thick,->,>=stealth,dotted] (B) edge (B1);
				\draw [thick,->,>=stealth,dotted] (C1) edge (C);
				\draw [thick,->,>=stealth,dashed] (B) edge (A);
				\draw [thick,->,>=stealth,dashed] (D) edge (C);
				\draw [thick,->,>=stealth,dashed] (A) edge [bend left=45] (D);
				\draw [thick,->,>=stealth,dashed] (C1) edge [bend left=45] (O1);
				\draw [thick,->,>=stealth,dashed] (E1) edge [bend left=45] (B1);
				
				\end{tikzpicture}
				
				\caption {Transition $1 \rightarrow 0$ (case 2 (e))}
				\label {Fig_1_0_forward_e}
			\end{figure}
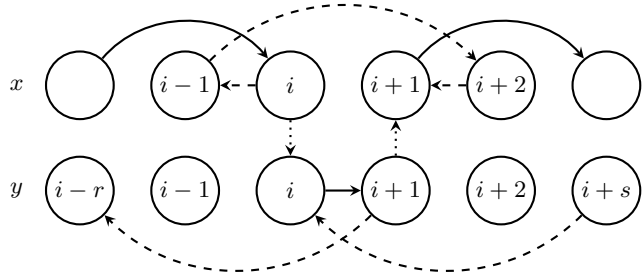
			We consider the edge $e_{i-1,y}$ of $y$ that leaves the city $i-1$.
			\begin{itemize}
				\item Let the edge $e_{i-1,y}$ be directed to a city with a larger number.
				Suppose that there exists a pyramidal tour with step-backs $t$ that contains both the edges $ (i+1) \leftarrow (i+2)$ and $(i) \leftarrow (i+s)$ that are not included in $z$ (Fig.~\ref {Fig_1_0_forward_e_1}, the edges of $t$ are solid).
				\begin{figure} [t]
				\centering
				\begin{tikzpicture}[scale=0.7]
				\begin{scope}[every node/.style={circle,thick,draw,inner sep=0pt, minimum size=0.9cm}]
				\node (A) at (0,0) {$i-1$};
				\node (B) at (2,0) {$i$};
				\node (C) at (4,0) {$i+1$};
				\node (D) at (6,0) {$i+2$};			
				\node (A1) at (0,-2) {$i-1$};
				\node (B1) at (2,-2) {$i$};
				\node (C1) at (4,-2) {$i+1$};
				\node (D1) at (6,-2) {};
				\node (E1) at (8,-2) {$i+s$};
				\end{scope}
				
				\node at (-1.2,0) {$x$};
				\node at (-1.2,-2) {$y$};
				
				\draw [thick,->,>=stealth,dashed] (B1) edge (C1);
				\draw [thick,->,>=stealth,dotted] (B1) edge (B);
				\draw [thick,->,>=stealth,dotted] (A) edge (A1);
				\draw [thick,->,>=stealth] (B) edge (A);
				\draw [thick,->,>=stealth] (D) edge (C);
				\draw [thick,->,>=stealth,dashed] (A) edge [bend left=40] (D);
				\draw [thick,->,>=stealth] (E1) edge [bend left=40] (B1);
				\draw [thick,->,>=stealth,dashed] (A1) edge [bend right=40] (D1);
				\end{tikzpicture}
				
				\caption {Case 2 (e), $e_{i-1,y}$ is directed to a city with a larger number}
				\label {Fig_1_0_forward_e_1}
			\end{figure}
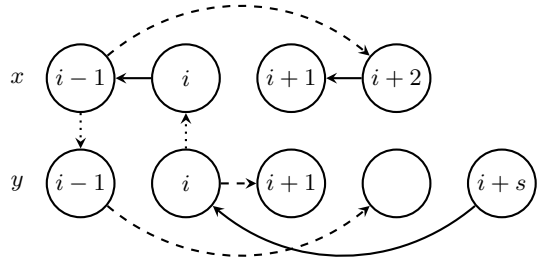
			There are two edges from $i$: $(i) \rightarrow (i+1)$ and $(i-1) \leftarrow (i)$.
			The edge $(i) \rightarrow (i+1)$ cannot be part of $t$, since in this case two edges $(i) \rightarrow (i+1)$ and $(i+1) \leftarrow (i+2) $ of $t$ enter the city $i+1$.
			Therefore, the edge $(i-1) \leftarrow (i)$ is a part of $t$, and the cities $i-1$ and $i$ are visited by $t$ in descending order.
			However, both edges leaving the city $i-1$ are directed to cities with numbers at least $i+1$.
			Thus, no pyramidal tour with step-backs $t$ can contain both edges $(i+1) \leftarrow (i+2)$ and $(i) \leftarrow (i+s)$ that are not included in $z$.
			By Lemma~\ref{lemma_zero_coefficient}, we got a contradiction.
			
			\item Let the edge $e_{i-1,y}$ be directed to a city with a smaller number.
			Suppose that there exists a pyramidal tour with step-backs $t$ that contains both the edges $(i-1) \leftarrow (i)$ and $(i-r) \leftarrow (i+1)$ that are not included in $z$ (Fig.~\ref {Fig_1_0_forward_e_2}, the edges of $t$ are solid).
			\begin{figure} [t]
				\centering
				\begin{tikzpicture}[scale=0.7]
				
				\begin{scope}[every node/.style={circle,thick,draw,inner sep=0pt, minimum size=0.9cm}]
				\node (A) at (0,0) {$i-1$};
				\node (B) at (2,0) {$i$};
				\node (C) at (4,0) {$i+1$};
				\node (D) at (6,0) {$i+2$};			
				\node (O1) at (-4,-2) {$i-r$};
				\node (P1) at (-2,-2) {$i-p$};
				\node (A1) at (0,-2) {$i-1$};
				\node (B1) at (2,-2) {$i$};
				\node (C1) at (4,-2) {$i+1$};
				\end{scope}

				\node at (-5.2,0) {$x$};
				\node at (-5.2,-2) {$y$};
				
				\draw [thick,->,>=stealth,dashed] (B1) edge (C1);
				\draw [thick,->,>=stealth,dotted] (A) edge (A1);
				\draw [thick,->,>=stealth,dotted] (C) edge (C1);
				\draw [thick,->,>=stealth] (B) edge (A);
				\draw [thick,->,>=stealth] (D) edge (C);
				\draw [thick,->,>=stealth] (A1) edge (P1);
				\draw [thick,->,>=stealth,dashed] (A) edge [bend left=45] (D);
				\draw [thick,->,>=stealth] (C1) edge [bend left=35] (O1);
				
				\end{tikzpicture}
				
				\caption {Case 2 (e), $e_{i-1,y}$ is directed to a city with a smaller number}
				\label {Fig_1_0_forward_e_2}
			\end{figure}
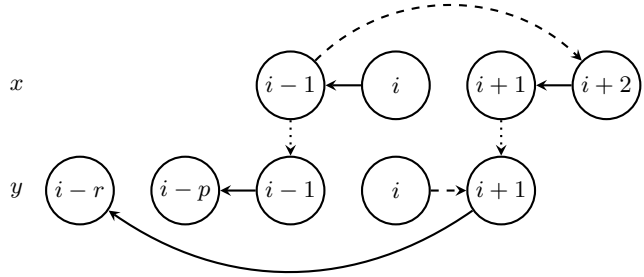
			There are two edges directed to $i+1$: $(i) \rightarrow (i+1)$ and $(i+1) \leftarrow (i+2)$.
			The edge $(i) \rightarrow (i+1)$ cannot be part of $t$, since in this case two edges $(i) \rightarrow (i+1)$ and $(i-1) \leftarrow (i)$ of $t$ leaves the city $i$.
			Therefore, the edge $(i+1) \leftarrow (i+2)$ is a part of $t$ and the cities $i+1$ and $i+2$ are visited by $t$ in descending order.
			In this case, the tour $t$ has the edge $(i-p) \leftarrow (i-1)$ of two edges that leave the city $i-1$. 
			Consequently, the cities $i$ and $i-1$ are also visited in descending order.
			However, no pyramidal tour with step-backs can go along the edges $(i+1) \leftarrow (i+2)$ and $(i-r) \leftarrow (i+1)$ and visit the city $i$ in descending order.
			Thus, no pyramidal tour with step-backs $t$ can contain both the edges $(i-1) \leftarrow (i)$ and $(i-r) \leftarrow (i+1)$ that are not included in $z$.
			By Lemma~\ref{lemma_zero_coefficient}, we got a contradiction.
			\end{itemize}
		
			\item[f)] Similar to the case c) there are only three edges: $(i-1) \leftarrow (i)$, $(i-r) \leftarrow (i+2)$, and $(i-s) \leftarrow (i+1)$ that lead to the cities with numbers less than $i$ (Fig.~\ref{Fig_1_0_forward_f}), and none of them are included in $z$.
			The tour $z$ cannot return to the city $1$.
			We got a contradiction.
			\begin{figure} [t]
				\centering
				\begin{tikzpicture}[scale=0.7]
				
				\begin{scope}[every node/.style={circle,thick,draw,inner sep=0pt, minimum size=0.9cm}]
				\node (O) at (-2,0) {$i-r$};
				\node (A) at (0,0) {$i-1$};
				\node (B) at (2,0) {$i$};
				\node (C) at (4,0) {$i+1$};
				\node (D) at (6,0) {$i+2$};
				\node (A1) at (0,-2) {$i-s$};
				\node (B1) at (2,-2) {$i$};
				\node (C1) at (4,-2) {$i+1$};	
				\end{scope}
				
				\node at (-3.2,0) {$x$};
				\node at (-3.2,-2) {$y$};
				
				\draw [thick,->,>=stealth] (B1) edge (C1);
				\draw [thick,->,>=stealth] (C) edge (D);
				\draw [thick,->,>=stealth,dashed] (B) edge (A);
				\draw [thick,->,>=stealth,dashed] (D) edge [bend right=35] (O);
				\draw [thick,->,>=stealth,dashed] (C1) edge [bend left=50] (A1);
				\draw [thick,->,>=stealth,dotted] (B) edge (B1);
				\draw [thick,->,>=stealth,dotted] (C1) edge (C);
				
				\end{tikzpicture}
				
				\caption {Transition $1 \rightarrow 0$ (case 2 (f))}
				\label {Fig_1_0_forward_f}
			\end{figure}
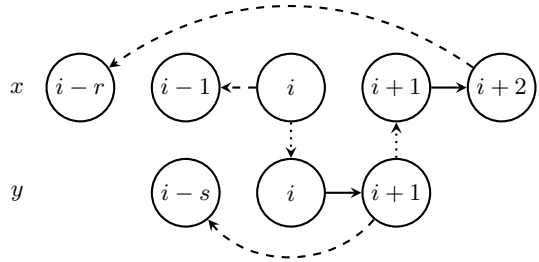
		\end{enumerate}
	\end{enumerate}
	
	Thus, if the tour $z$ enters the city $i$ along the edge of $x$ and leaves along the edge of $y$, or vice versa, then $i$ is visited by $x$ and $y$ in the same order.
	Since, from the city 1 there are only edges leading to the cities that are visited in ascending order by $x$ and $y$, the tour $z$ in ascending order includes only edges of ascending orders of $x$ and $y$, in descending order -- only edges of descending orders.

	\textit{Part 2.} Let us prove that transitions between the edges of $x$ and $y$ can only be performed by blocks $U,L,R$ from the statement of the theorem.
	Without loss of generality, we assume that the city $i$ is visited by $x$ and $y$ in ascending order ($x^{0,1}_i = y^{0,1}_i = 1$). The case $x^{0,1}_i = y^{0,1}_i = 0$ is treated similarly.
	We consider the possible configurations:
	\begin{gather*}
	\text{a)} \left\langle \begin{matrix} 1 \\ 1 \end{matrix} \right\rangle,\
	\text{b)} \left\langle \begin{matrix} 1 \hspace{1.2em} \\ \lvv {1 \ \ 1}\end{matrix} \right\rangle,\
	\text{c)} \left\langle \begin{matrix} \hspace{1.2em} 1 \\ \lvv {1 \ \ 1} \end{matrix} \right\rangle,\
	\text{d)} \left\langle \begin{matrix} \lvv {1 \ \ 1} \\ \hspace{1.2em} 1 \end{matrix} \right\rangle,\
	\text{e)} \left\langle \begin{matrix} \lvv {1 \ \ 1} \\ 1 \hspace{1.2em} \end{matrix} \right\rangle,\\
	\text{f)} \left\langle \begin{matrix} \lvv {1 \ \ 1} \\ \lvv {1 \ \ 1} \end{matrix} \right\rangle,\
	\text{g)} \left\langle \begin{matrix} \hspace{1.2em} \lvv {1 \ \ 1} \\ \lvv {1 \ \ 1} \hspace{1.2em} \end{matrix} \right\rangle,\
	\text{h)} \left\langle \begin{matrix} \lvv {1 \ \ 1} \hspace{1.2em} \\ \hspace{1.2em} \lvv {1 \ \ 1} \end{matrix} \right\rangle.
	\end{gather*}
	
	\begin{enumerate}
		\item[a)] The transition has the form of a block $U_{11}$.
		
		\item[b)] Suppose that the city $i+1$ is visited by $x$ in ascending order ($x^{0,1}_{i+1} = 1$):
		\[ 
		\text{b1)} \left\langle \begin{matrix} 1 \ \ 1 \\ \lvv {1 \ \ 1}\end{matrix} \right\rangle,\
		\text{b2)} \left\langle \begin{matrix} 1 \ \ \lvv {1 \ \ 1} \\ \lvv {1 \ \ 1} \hspace{1.2em}\end{matrix}  \right\rangle.
		\]
		In the case b1) none of the edges entering the city $i+1$ was included in the tour $z$, therefore, $z$ cannot be a Hamiltonian tour.
		
		In the case b2), the tour $z$ can visit the city $i+1$ only along the edge $(i+1) \leftarrow (i+2)$ of $x$.
		At the same time, $z$ can enter the city $i+2$ only along the edge of the tour $y$ in ascending order:
		\[ 
		\text{b21)} \left\langle \begin{matrix} 1 \ \ \lvv {1 \ \ 1} \\ \lvv {1 \ \ 1} \ \ 1\end{matrix}  \right\rangle,\
		\text{b22)} \left\langle \begin{matrix} 1 \ \ \lvv {1 \ \ 1} \hspace{1.2em} \\ \lvv {1 \ \ 1} \ \ \lvv {1 \ \ 1}\end{matrix}  \right\rangle.
		\]
		In the case b21), transition between the tours $x$ and $y$ does not make sense, since the edge $(i) \rightarrow (i+2)$ is contained in both tours (Fig.~\ref{Fig_1_1_bg}).
		
		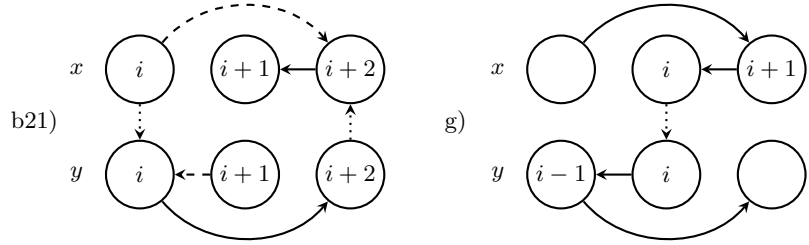
\begin{figure} [t]
			\centering
			\begin{tikzpicture}[scale=0.7]
			
			\begin{scope}[every node/.style={circle,thick,draw,inner sep=0pt, minimum size=0.9cm}]
			\node (A) at (0,0) {$i$};
			\node (B) at (2,0) {$i+1$};
			\node (C) at (4,0) {$i+2$};
			\node (A1) at (0,-2) {$i$};
			\node (B1) at (2,-2) {$i+1$};
			\node (C1) at (4,-2) {$i+2$};	
			\end{scope}
			
			\node at (-1.2,0) {$x$};
			\node at (-1.2,-2) {$y$};
			\node at (-2,-1) {b21)};
			
			\draw [thick,->,>=stealth,dashed] (A) edge [bend left=50] (C);
			\draw [thick,->,>=stealth] (C) edge (B);
			\draw [thick,->,>=stealth] (A1) edge [bend right=50] (C1);
			\draw [thick,->,>=stealth,dashed] (B1) edge (A1);
			\draw [thick,->,>=stealth,dotted] (A) edge (A1);
			\draw [thick,->,>=stealth,dotted] (C1) edge (C);
			
			\begin{scope}[xshift=8cm]
			\begin{scope}[every node/.style={circle,thick,draw,inner sep=0pt, minimum size=0.9cm}]
			\node (A) at (0,0) {};
			\node (B) at (2,0) {$i$};
			\node (C) at (4,0) {$i+1$};
			\node (A1) at (0,-2) {$i-1$};
			\node (B1) at (2,-2) {$i$};
			\node (C1) at (4,-2) {};	
			\end{scope}
			
			\node at (-1.2,0) {$x$};
			\node at (-1.2,-2) {$y$};
			\node at (-2,-1) {g)};
			
			\draw [thick,->,>=stealth] (A) edge [bend left=50] (C);
			\draw [thick,->,>=stealth] (C) edge (B);
			\draw [thick,->,>=stealth] (B1) edge (A1);
			\draw [thick,->,>=stealth] (A1) edge [bend right=50] (C1);
			\draw [thick,->,>=stealth,dotted] (B) edge (B1);
			\end{scope}
			
			\end{tikzpicture}
			
			\caption {Transition $1 \rightarrow 1$ (cases (b21) and (g))}
			\label {Fig_1_1_bg}
		\end{figure}
		The case b22) contains a transition of the configuration g), the impossibility of which will be considered separately.
		
		Thus, the city $i+1 $ can be visited by $x$ only in descending order ($x^{0,1}_{i+1} = 0$). 
		The transition has the form of a block:
		\[L_{1011} = \left\langle \begin{matrix} 1 \ \ \tilde{0} \\ \lvv{1 \ \ 1} \end{matrix} \right\rangle.\]
		
		\item[c,d,e)] Similar to configuration b), the transitions have the form of the blocks $R_{0111}$, $R_{1101}$, and $L_{1110}$.
		
		\item[f)] The transition has the form of a block $U_{1111}$.

		\item[g)] The tour $z$ enters the city $i$ by the edge $(i) \leftarrow (i+1)$ of $x$ and leaves by the edge $(i-1) \leftarrow (i)$ of $y$. 
		None of the pyramidal tours with step-backs can contain both these edges in ascending order, since this is a double step-back (Fig.~\ref{Fig_1_1_bg}).	
		
		\item[h)] The edges $(i-1) \leftarrow (i)$ and $(i) \leftarrow (i+1)$ are not included in $z$ and no pyramidal tour with step-backs can contain both these edges at the same time (Fig.~\ref{Fig_1_1_h}).
		By Lemma~\ref{lemma_zero_coefficient}, we got a contradiction.
		\begin{figure} [t]
			\centering
			\begin{tikzpicture}[scale=0.7]
			
			\begin{scope}[every node/.style={circle,thick,draw,inner sep=0pt, minimum size=0.9cm}]
			\node (O) at (-2,0) {};
			\node (A) at (0,0) {$i-1$};
			\node (B) at (2,0) {$i$};
			\node (B1) at (2,-2) {$i$};
			\node (C1) at (4,-2) {$i+1$};
			\node (D1) at (6,-2) {};	
			\end{scope}
			
			\node at (-3.2,0) {$x$};
			\node at (-3.2,-2) {$y$};
			
			\draw [thick,->,>=stealth] (O) edge [bend left=50] (B);
			\draw [thick,->,>=stealth,dashed] (B) edge (A);
			\draw [thick,->,>=stealth,dashed] (C1) edge (B1);
			\draw [thick,->,>=stealth] (B1) edge [bend right=50] (D1);
			\draw [thick,->,>=stealth,dotted] (B) edge (B1);
			
			\end{tikzpicture}
			
			\caption {Transition $1 \rightarrow 1$ (case (h))}
			\label {Fig_1_1_h}
		\end{figure}
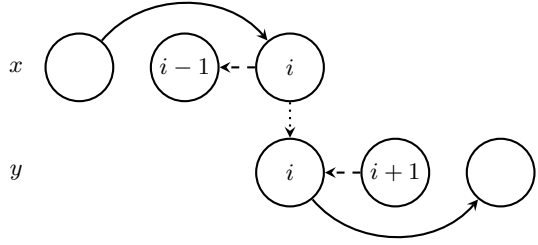
	\end{enumerate}

	Thus, the transition between the edges of $x$ and $y$ is possible only at blocks $U,L,R$ from the statement of the theorem. Besides, it can be done at cities $1$ and $n$ which can also be considered as universal blocks.

	\textit{Part 3.} Let us prove that the remaining conditions of the theorem are satisfied.
	By construction, $i$ is the city with the smallest number, such that $z$ enters $i$ by an edge of $x$ and leaves by an edge of $y$, $j$ is the city with the smallest number, such that $z$ enters $j$ by an edge of $y$ and leaves by an edge of $x$.
	
	The coordinates $x^{0,1}_{i}$ and $x^{0,1}_j$ can take one of the four combinations of the values $0$ and $1$ that are described in the statement of the theorem.
	Without loss of generality we assume that $i < j$ and $x^{0,1}_{i} = 1$, other cases are treated similarly.
	
	Since $i$ and $j$ are the cities with the smallest numbers where the tour $z$ makes transitions between edges of $x$ and $y$, the traversal diagram has the form shown in Fig.~\ref{Fig_3}. In particular, the tour $z$ visits $i-1$ by the edges of $x$.
	
	\begin{figure} [t]
		\centering
		\begin{tikzpicture}[scale=0.8]
		
		\begin{scope}[every node/.style={circle,thick,draw,inner sep=0pt, minimum size=0.9cm}]
		\node (A) at (0,0) {$i$};
		\node (B) at (4,0) {$j$};
		\node (A1) at (0,-2) {$i$};
		\node (B1) at (4,-2) {$j$};	
		\end{scope}
		
		\node at (-3.2,0) {$x$};
		\node at (-3.2,-2) {$y$};
		
		\draw [thick,<->,>=stealth] (A) edge (-2,0);
		\draw [thick,->,>=stealth] (B) edge (A);
		\draw [thick,->,>=stealth] (A1) edge (B1);
		\draw [thick,->,>=stealth,dotted] (A) edge (A1);
		\draw [thick,->,>=stealth,dotted] (B1) edge (B);
		
		\end{tikzpicture}
		
		\caption {The traversal diagram of the tour $z$ if $i < j$}
		\label {Fig_3}
	\end{figure}
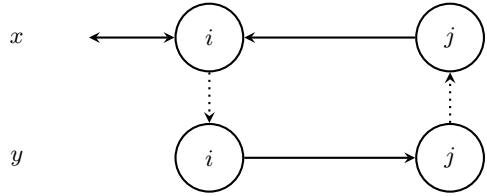
	
	First we prove that transition at $i$ has the form of a left block $U$ or $L$.
	Suppose the contrary, then the block has one of two possible forms:
	\[ 
	R_{1101} = \left\langle \begin{matrix} \lvv {1 \ \ 1} \\ \tilde{0} \ \ 1 \end{matrix} \right\rangle,\
	R_{0111} = \left\langle \begin{matrix} \tilde{0} \ \ 1 \\ \lvv {1 \ \ 1} \end{matrix} \right\rangle.
	\]
	In the case of $R_{1101}$, the tour $z$ skips the city $i-1$ in ascending order, since the edge $(i-1) \leftarrow (i)$ of $x$ is not a part of $z$, and then again skips $i-1$ in descending order.
	In the case of $R_{0111}$, the tour $z$ visits the city $i-1$ in ascending order, since this time $(i-1) \leftarrow (i)$ of $y$ is a part of $z$, and then again visits $i-1$ in descending order.
	In both cases, the tour is not Hamiltonian.
	Similarly, we can prove that the transition at $j$ cannot have the form of a left block $L$.
	
	We denote by $i_a$ the first city after the left block: $i_a = i+1$ for single blocks and $i_a = i + 2$ for double blocks. We denote by $j_b$ the last city before the right block: $j_b = i-1$ for single blocks and $j_b = j-2$ for double blocks.
	
	Therefore, by construction, in the central part between the blocks, the tour $z$ in ascending order  goes along the edges of $y$: $z^{1,sb}_{[i_a,j_b]} = y^{1,sb}_{[i_a,j_b]}$, in descending order -- along the edges of $x$: $z^{0,sb}_{[i_a,j_b]} = x^{0,sb}_{[i_a,j_b]}$.  While on the left side $z$ moves along the edges of $x$ in both directions: $z^{0,1,sb}_{[1,i_a-1]} = x^{0,1,sb}_{[1,i_a-1]}$ (Fig~\ref{Fig_3}).

	We combine the conditions for the central part:
	\[
	\begin{cases}
	z^{1,sb}_{[i_a,j_b]} = y^{1,sb}_{[i_a,j_b]},\\
	z^{0,sb}_{[i_a,j_b]} = x^{0,sb}_{[i_a,j_b]}
	\end{cases}
	\Rightarrow \ \ \
	x^{0,1}_{[i_a,j_b]} = y^{0,1}_{[i_a,j_b]}.
	\]	
	Indeed, if for at least one city in the central part the coordinates of $x^{0,1}_{[i_a,j_b]}$ and $y^{0,1}_{[i_a,j_b]}$ do not match, then the tour $z$ will either skip this city or visit it twice.
	
	\begin{enumerate}
		\item If $x^{0,1}_{j} = 1$, then both cities $i$ and $j$ are visited by $x,y,z$ in ascending order.
		We verify the remaining conditions of the theorem.
		\begin{itemize}
			\item If the first condition is not satisfied:
			\[x^{1,sb}_{[i_a,j_b]} = y^{1,sb}_{[i_a,j_b]} = z^{1,sb}_{[i_a,j_b]},\] 
			then the transitions at the cities $i$ and $j$ do not make sense, since all the edges of $y$ that are part of $z$ as a result are also contained in the tour $x$.
			
			\item If the second condition is not satisfied:
			\[
			\begin{cases}
			x^{0,sb}_{[i_a,j_b]} = y^{0,sb}_{[i_a,j_b]} = z^{0,sb}_{[i_a,j_b]},\\
			x^{0,1,sb}_{[1,i_a-1]} = y^{0,1,sb}_{[1,i_a-1]} = z^{0,1,sb}_{[1,i_a-1]},\\
			x^{0,1,sb}_{[j_b+1,n]} = y^{0,1,sb}_{[j_b+1,n]} = z^{0,1,sb}_{[j_b+1,n]},
			\end{cases}
			\] 
			then the tour $z$ completely coincides with the tour $y$.
		\end{itemize}
		
		\item If $x^{0,1}_{j} = 0$, then the city $i$ is visited by $x,y,z$ in ascending order, the city $j$ -- in descending order. We verify the remaining conditions of the theorem.
		\begin{itemize}
			\item If the first condition is not satisfied:
			\[
			\begin{cases}
			x^{1,sb}_{[i_a,j_b]} = y^{1,sb}_{[i_a,j_b]} = z^{1,sb}_{[i_a,j_b]},\\
			x^{0,1,sb}_{[j_b+1,n]} = y^{0,1,sb}_{[j_b+1,n]} = z^{0,1,sb}_{[j_b+1,n]},
			\end{cases}
			\]
			then the tour $z$ completely coincides with the tour $x$.
			
			\item If the second condition is not satisfied:
			\[
			\begin{cases}
			x^{0,sb}_{[i_a,j_b]} = y^{0,sb}_{[i_a,j_b]} = z^{0,sb}_{[i_a,j_b]},\\
			x^{0,1,sb}_{[1,i_a-1]} = y^{0,1,sb}_{[1,i_a-1]} = z^{0,1,sb}_{[1,i_a-1]},
			\end{cases}
			\]	
			then the transitions at the cities $i$ and $j$ do not make sense, since all the edges of $x$ that are part of $z$ as a result are also contained in the tour $y$.	
		\end{itemize}
	\end{enumerate}
	
	Thus, if the vertices $x^v$ and $y^v$ of the polytope $\PSB(n)$ are not adjacent, then the conditions of the theorem are satisfied.
	
	\textit {Sufficiency.}
	Suppose that sufficient conditions of the theorem are satisfied.
	We consider the pyramidal tour with step-backs $z$, constructed as described in Table~\ref{table_tour_z}, and the pyramidal tour with step-backs $t$, constructed as $t = (x \cup y) \backslash z$.
	The multigraph $x \cup y$ includes a pair of complementary pyramidal tours with step-backs $z$ and $t$, different from $x$ and $y$. Thus, by Lemma~\ref{lemma_sufficient} the vertices $x^v$ and $y^v$ of the polytope $\PSB(n)$ are not adjacent. Examples of the first and third sufficient conditions are shown in Fig.~\ref{Fig_sufficient}.
	
	\begin{table}[t]
		\centering
		\caption{Construction of the tour $z$}
		\label{table_tour_z}
		\begin{tabular}{|c|c|}
			\hline
			\begin{minipage}{0.4\linewidth}
				~\\
				1. If $x^{0,1}_{i} = x^{0,1}_{j} = 1$, then
				\[z^{0,1,sb}_{k} =
				\begin{cases}
				x^{0,1,sb}_{k},& \text{if } k < i_a,\\
				y^{1,sb}_{k},& \text{if } i_a \leq k \leq j_b,\\
				x^{0,sb}_{k},& \text{if } i_a \leq k \leq j_b,\\
				x^{0,1,sb}_{k},& \text{if } k > j_b.
				\end{cases}\]
			\end{minipage} & 
			\begin{minipage}{0.4\linewidth}
				~\\
				2. If $x^{0,1}_{i} = x^{0,1}_{j} = 0$, then 
				\[z^{0,1,sb}_{k} =
				\begin{cases}
				x^{0,1,sb}_{k},& \text{if } k < i_a,\\
				x^{1,sb}_{k},& \text{if } i_a \leq k \leq j_b,\\
				y^{0,sb}_{k},& \text{if } i_a \leq k \leq j_b,\\
				x^{0,1,sb}_{k},& \text{if } k > j_b.
				\end{cases}\]
			\end{minipage}
			\\ \hline
			\begin{minipage}{0.4\linewidth}
				~\\
				3. If $x^{0,1}_{i} = 1$, $x^{0,1}_{j} = 0$, then 
				\[z^{0,1,sb}_{k} =
				\begin{cases}
				x^{0,1,sb}_{k},& \text{if } k < i_a,\\
				x^{0,sb}_{k},& \text{if } i_a \leq k \leq j_b,\\
				y^{1,sb}_{k},& \text{if } i_a \leq k \leq j_b,\\
				y^{0,1,sb}_{k},& \text{if } k > j_b.
				\end{cases}\]
			\end{minipage} & 
			\begin{minipage}{0.4\linewidth}
				~\\
				4. If $x^{0,1}_{i} = 0$, $x^{0,1}_{j} = 1$, then 
				\[z^{0,1,sb}_{k} =
				\begin{cases}
				x^{0,1,sb}_{k},& \text{if } k < i_a,\\
				x^{1,sb}_{k},& \text{if } i_a \leq k \leq j_b,\\
				y^{0,sb}_{k},& \text{if } i_a \leq k \leq j_b,\\
				y^{0,1,sb}_{k},& \text{if } k > j_b.
				\end{cases}\]
		\end{minipage} \\ \hline
		\end{tabular}
	\end{table}
	
	\begin{figure} [t]
		\centering
		\begin{tikzpicture}[scale=0.8]
		
		\node at (-0.7,0) {\textit{x}};
		\node at (0,0) {$ \langle $};
		\node at (1,0) {1};
		\node at (2,0) {1};
		\node at (3,0) {1};
		\node at (4,0) {1};
		\node at (5,0) {1};
		\node at (6,0) {$ \rangle $};
		
		\draw [<-,>=stealth] (1.85,0.25) -- (3.1,0.25);
		
		\node at (-0.7,-0.75) {\textit{y}};
		\node at (0,-0.75) {$ \langle $};
		\node at (1,-0.75) {1};
		\node at (2,-0.75) {1};
		\node at (3,-0.75) {1};
		\node at (4,-0.75) {1};
		\node at (5,-0.75) {0};
		\node at (6,-0.75) {$ \rangle $};
		
		\draw [dashed] (0.75,0.25) -- (1.25,0.25) -- (1.25,-1) -- (0.75,-1) -- cycle;
		\draw [dashed] (3.75,0.25) -- (4.25,0.25) -- (4.25, -1) -- (3.75,-1) -- cycle;
		
		\begin{scope}[yshift=-2cm]
		\begin{scope}[every node/.style={circle,thick,draw,inner sep=2.4pt}]
		\node (A) at (0,0) {1};
		\node (B) at (1,0) {2};
		\node (C) at (2,0) {3};
		\node (D) at (3,0) {4};
		\node (E) at (4,0) {5};
		\node (F) at (5,0) {6};
		\node (G) at (6,0) {7};
		\end{scope}
		\draw [thick,->,>=stealth] (A) edge (B);
		\draw [thick,->,>=stealth] (B) edge [bend left=50] (D);
		\draw [thick,->,>=stealth] (D) edge (C);	
		\draw [thick,->,>=stealth] (C) edge [bend left=50] (E);
		\draw [thick,->,>=stealth] (E) edge (F);	
		\draw [thick,->,>=stealth] (F) edge (G);
		\draw [thick,->,>=stealth] (G) edge [bend left=30] (A);
		\draw (-0.7, 0) node{\textit{x}};
		\end{scope}
		
		\begin{scope}[yshift=-4cm]
		\begin{scope}[every node/.style={circle,thick,draw,inner sep=2.4pt}]
		\node (A) at (0,0) {1};
		\node (B) at (1,0) {2};
		\node (C) at (2,0) {3};
		\node (D) at (3,0) {4};
		\node (E) at (4,0) {5};
		\node (F) at (5,0) {6};
		\node (G) at (6,0) {7};
		\end{scope}
		\draw [thick,->,>=stealth,dashed] (A) edge (B);
		\draw [thick,->,>=stealth,dashed] (B) edge (C);
		\draw [thick,->,>=stealth,dashed] (C) edge (D);	
		\draw [thick,->,>=stealth,dashed] (D) edge (E);	
		\draw [thick,->,>=stealth,dashed] (E) edge [bend left=50] (G);
		\draw [thick,->,>=stealth,dashed] (G) edge (F);		
		\draw [thick,->,>=stealth,dashed] (F) edge [bend left=35] (A);
		\draw (-0.7, 0) node{\textit{y}};		
		\end{scope}
		
		\begin{scope}[yshift=-6cm]
		\begin{scope}[every node/.style={circle,thick,draw,inner sep=2.4pt}]
		\node (A) at (0,0) {1};
		\node (B) at (1,0) {2};
		\node (C) at (2,0) {3};
		\node (D) at (3,0) {4};
		\node (E) at (4,0) {5};
		\node (F) at (5,0) {6};
		\node (G) at (6,0) {7};
		\end{scope}
		\draw [thick,->,>=stealth] (A) edge (B);
		\draw [thick,->,>=stealth,dashed] (B) edge (C);
		\draw [thick,->,>=stealth,dashed] (C) edge (D);	
		\draw [thick,->,>=stealth,dashed] (D) edge (E);
		\draw [thick,->,>=stealth] (E) edge (F);	
		\draw [thick,->,>=stealth] (F) edge (G);
		\draw [thick,->,>=stealth] (G) edge [bend left=30] (A);	
		\draw (-0.7, 0) node{\textit{z}};		
		\end{scope}
		
		\begin{scope}[yshift=-8cm]
		\begin{scope}[every node/.style={circle,thick,draw,inner sep=2.4pt}]
		\node (A) at (0,0) {1};
		\node (B) at (1,0) {2};
		\node (C) at (2,0) {3};
		\node (D) at (3,0) {4};
		\node (E) at (4,0) {5};
		\node (F) at (5,0) {6};
		\node (G) at (6,0) {7};
		\end{scope}
		\draw [thick,->,>=stealth,dashed] (A) edge (B);
		\draw [thick,->,>=stealth] (B) edge [bend left=50] (D);
		\draw [thick,->,>=stealth] (D) edge (C);	
		\draw [thick,->,>=stealth] (C) edge [bend left=50] (E);
		\draw [thick,->,>=stealth,dashed] (E) edge [bend left=50] (G);
		\draw [thick,->,>=stealth,dashed] (G) edge (F);		
		\draw [thick,->,>=stealth,dashed] (F) edge [bend left=35] (A);
		\draw (-0.7, 0) node{\textit{t}};		
		\end{scope}
		
		\begin{scope}[xshift=8cm]
		\node at (-0.7,0) {\textit{x}};
		\node at (0,0) {$ \langle $};
		\node at (1,0) {1};
		\node at (2,0) {0};
		\node at (3,0) {0};
		\node at (4,0) {1};
		\node at (5,0) {0};
		\node at (6,0) {$ \rangle $};	
		
		\node at (-0.7,-0.75) {\textit{y}};
		\node at (0,-0.75) {$ \langle $};
		\node at (1,-0.75) {1};
		\node at (2,-0.75) {0};
		\node at (3,-0.75) {0};
		\node at (4,-0.75) {0};
		\node at (5,-0.75) {0};
		\node at (6,-0.75) {$ \rangle $};
		
		\draw [->,>=stealth] (1.9,-0.5) -- (3.1,-0.5);
		\draw [->,>=stealth] (3.9,-0.5) -- (5.1,-0.5);
		
		\draw [dashed] (0.75,0.25) -- (1.25,0.25) -- (1.25,-1) -- (0.75,-1) -- cycle;
		\draw [dashed] (3.75,0.25) -- (5.25,0.25) -- (5.25, -1) -- (3.75,-1) -- cycle;	
		
		\begin{scope}[yshift=-2cm]
		\begin{scope}[every node/.style={circle,thick,draw,inner sep=2.4pt}]
		\node (A) at (0,0) {1};
		\node (B) at (1,0) {2};
		\node (C) at (2,0) {3};
		\node (D) at (3,0) {4};
		\node (E) at (4,0) {5};
		\node (F) at (5,0) {6};
		\node (G) at (6,0) {7};
		\end{scope}
		\draw [thick,->,>=stealth] (A) edge (B);
		\draw [thick,->,>=stealth] (B) edge [bend left=40] (E);
		\draw [thick,->,>=stealth] (E) edge [bend left=50] (G);
		\draw [thick,->,>=stealth] (G) edge (F);
		\draw [thick,->,>=stealth] (F) edge [bend left=50] (D);
		\draw [thick,->,>=stealth] (D) edge (C);
		\draw [thick,->,>=stealth] (C) edge [bend left=50] (A);
		\draw (-0.7, 0) node{\textit{x}};			
		\end{scope}
		
		\begin{scope}[yshift=-4cm]
		\begin{scope}[every node/.style={circle,thick,draw,inner sep=2.4pt}]
		\node (A) at (0,0) {1};
		\node (B) at (1,0) {2};
		\node (C) at (2,0) {3};
		\node (D) at (3,0) {4};
		\node (E) at (4,0) {5};
		\node (F) at (5,0) {6};
		\node (G) at (6,0) {7};
		\end{scope}
		\draw [thick,->,>=stealth,dashed] (A) edge (B);
		\draw [thick,->,>=stealth,dashed] (B) edge [bend left=33] (G);
		\draw [thick,->,>=stealth,dashed] (G) edge [bend left=50] (E);	
		\draw [thick,->,>=stealth,dashed] (E) edge (F);	
		\draw [thick,->,>=stealth,dashed] (F) edge [bend left=45] (C);
		\draw [thick,->,>=stealth,dashed] (C) edge (D);	
		\draw [thick,->,>=stealth,dashed] (D) edge [bend left=45] (A);
		\draw (-0.7, 0) node{\textit{y}};		
		\end{scope}
		
		\begin{scope}[yshift=-6cm]
		\begin{scope}[every node/.style={circle,thick,draw,inner sep=2.4pt}]
		\node (A) at (0,0) {1};
		\node (B) at (1,0) {2};
		\node (C) at (2,0) {3};
		\node (D) at (3,0) {4};
		\node (E) at (4,0) {5};
		\node (F) at (5,0) {6};
		\node (G) at (6,0) {7};
		\end{scope}
		\draw [thick,->,>=stealth] (A) edge (B);
		\draw [thick,->,>=stealth,dashed] (B) edge [bend left=33] (G);
		\draw [thick,->,>=stealth,dashed] (G) edge [bend left=50] (E);	
		\draw [thick,->,>=stealth,dashed] (E) edge (F);	
		\draw [thick,->,>=stealth] (F) edge [bend left=50] (D);
		\draw [thick,->,>=stealth] (D) edge (C);
		\draw [thick,->,>=stealth] (C) edge [bend left=50] (A);
		\draw (-0.7, 0) node{\textit{z}};		
		\end{scope}
		
		\begin{scope}[yshift=-8cm]
		\begin{scope}[every node/.style={circle,thick,draw,inner sep=2.4pt}]
		\node (A) at (0,0) {1};
		\node (B) at (1,0) {2};
		\node (C) at (2,0) {3};
		\node (D) at (3,0) {4};
		\node (E) at (4,0) {5};
		\node (F) at (5,0) {6};
		\node (G) at (6,0) {7};
		\end{scope}
		\draw [thick,->,>=stealth,dashed] (A) edge (B);
		\draw [thick,->,>=stealth] (B) edge [bend left=45] (E);
		\draw [thick,->,>=stealth] (E) edge [bend left=50] (G);
		\draw [thick,->,>=stealth] (G) edge (F);
		\draw [thick,->,>=stealth,dashed] (F) edge [bend left=45] (C);
		\draw [thick,->,>=stealth,dashed] (C) edge (D);	
		\draw [thick,->,>=stealth,dashed] (D) edge [bend left=45] (A);
		\draw (-0.7, 0) node{\textit{t}};		
		\end{scope}
		\end{scope}
		\end{tikzpicture}	
		\caption{Examples of first and third sufficient conditions}
		\label{Fig_sufficient}
	\end{figure}
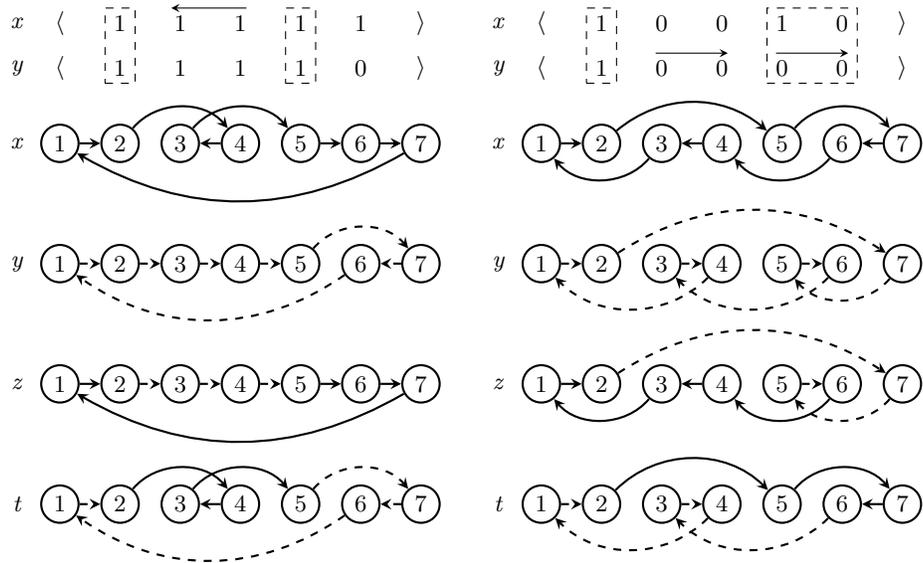

\end{proof}

\begin{theorem}
	The question whether two vertices of the polytope $\PSB(n)$ are adjacent can be verified in polynomial time.	
\end{theorem}

\begin{proof}
	We consider two pyramidal tours with step-backs $x$ and $y$, and the corresponding vertices $x^v$ and $y^v$ of the polytope $\PSB(n)$.
	
	In the encodings $x^{0,1,sb}$ and $y^{0,1,sb}$ there are $O(n)$ possible positions for the left block. Similarly, there are $O(n)$ possible positions for the right block. For each pair of blocks, the verification of the remaining conditions will require a single pass along the vectors $x^{0,1,sb}$ and $y^{0,1,sb}$ that can be performed in time $O(n)$. Thus, the vertex adjacency test by an exhaustive search of all possible cases of the Theorem~\ref{theorem_adjacency} will require at most $O(n^3)$ operations.
	
	In fact, the test can be performed in linear time $O(n)$.
	A single pass through the encodings $x^{0,1,sb}$ and $y^{0,1,sb}$ is enough to consistently find the left block, then the right block, then check the remaining conditions.
\end{proof}

\section{Conclusion}

The general formulation of the traveling salesperson problem and the verification of vertex adjacency in 1-skeleton of the traveling salesperson polytope are NP-complete \cite{Papadimitriou}.
At the same time, the traveling salesperson problem for pyramidal tours and pyramidal tours with step-backs is solvable by dynamic programming in polynomial time \cite{Enomoto,Gilmore}.
We have established that the vertex adjacency in 1-skeleton of the polytope of pyramidal tours \cite {Bondarenko-Nikolaev-2018} and pyramidal tours with step-backs can be verified in polynomial time.
Thus, the properties of 1-skeleton of the traveling salesperson polytope are directly related to the properties of the problem itself.

\subsubsection*{Acknowledgments.} The research is supported by the grant of the President of the Russian Federation MK-2620.2018.1.

%
%
%
%

\end{document}